\documentclass{amsart}
\usepackage{amsfonts}
\usepackage{amsmath}
\usepackage{amssymb}
\usepackage{amsthm}
\usepackage{amscd}
\usepackage{graphicx}
\usepackage{url}

\newtheorem{thm}{Theorem}[section]
\newtheorem{prop}[thm]{Proposition}

\newtheorem{lemma}[thm]{Lemma}
\newtheorem{cor}[thm]{Corollary}
\newtheorem{conj}{Conjecture}

\def\be{\begin{equation}}
\def\ee{\end{equation}}
\def\ba{\begin{array}}
\def\ea{\end{array}}
\def\bee{\begin{eqnarray}}
\def\eee{\end{eqnarray}}

\def\P{{\cal P}}
\DeclareMathOperator{\im}{im}

\renewcommand{\P}{\mathcal{D}}  
\newcommand{\D}{\mathcal{D}}  
\renewcommand{\H}{\mathcal{H}}  

\title{A proof of the Kauffman--Harary Conjecture}
\author{Thomas W.\ Mattman and Pablo Solis}
\address{Department of Mathematics and Statistics,
         California State University, Chico,
         Chico, CA 95929-0525}
\email{TMattman@CSUChico.edu}
\address{Department of Mathematics,
 University of California,
 Berkeley, CA}
\email{pablo@math.berkeley.edu}


\begin{document}

\begin{abstract}
We prove the Kauffman--Harary Conjecture,
posed in 1999: given a reduced, alternating diagram $\P$ of a knot with prime determinant $p$, every
non-trivial Fox $p$-coloring of $\P$ will assign different colors to different arcs.
\end{abstract}

\maketitle

\section{Introduction}\label{sec-intro}
In 1999, Louis Kauffman and Frank Harary published a paper
\cite{Kauff} detailing a graph-theoretic approach to the study of
knot theory. In the paper they state a conjecture
(Alternation Conjecture 6.2) that has come to be known as
the Kauffman--Harary Conjecture:
\begin{conj}[Kauffman--Harary Conjecture]\label{k-h conj}
Let $\P$ be a reduced, alternating diagram of the knot $k$ having prime determinant $p$.
Then, every non-trivial $p$-coloring of $\P$ assigns different colors to different arcs.
\end{conj}
This provides a nice connection between two knot invariants, the 
determinant, which is relatively easy to calculate, for example, using the
Alexander polynomial, and the least number of colors required on a minimal diagram of a knot, an invariant which is, in general, very difficult to evaluate. The conjecture asserts that in the case of an alternating knot,
if the determinant is prime, then the least number of colors is simply the crossing number.
Conjecture~1 is known to hold for rational knots~\cite{KL,PDDGS}, pretzel knots~\cite{Asaeda},
and many Turk's head knots~\cite{DMMS}. Our goal in this paper is to prove the conjecture for all knots. 

In order to give an overview of our approach we recall some
basic ideas about coloring. Let $k$ be an alternating knot of prime
determinant $p$.  A (Fox) $p$-coloring of a diagram $\P$ of $k$ is
a way of coloring the arcs of $\P$ such that the equation
$2x - y - z \equiv 0$ mod $p$ holds at each crossing, $x$ being the color of
the over arc, while $y$ and $z$ are the colors of the two under arcs.
It's convenient to encode this requirement as an $n \times n$ matrix $C'$,
the crossing matrix,
where $n$ is the number of crossings in $\P$. Each row of $C'$ corresponds
to a crossing and has one $2$ entry and two $-1$'s, the other entries being
$0$. The columns of $C'$ then correspond to arcs of $\P$ and colorings
to vectors $X$ such that $C'X \equiv \vec 0$ modulo $p$.

Now, every constant vector $X = (c,c,c,\dotsc,c)$ gives a coloring,
but if we move to the minor $C$ defined by removing the last row and
column of $C'$, we will have a way of identifying non-trivial
colorings. It turns out that $|\det C| = \det k$, the knot's determinant, which is a knot invariant. As $\det k = p$, a prime, $C$ is
invertible over the rationals. 
We define the coloring matrix $L = p C^{-1}$ as the
classical adjoint of $C$. In Section~\ref{sec:main result on L}, we 
argue that each column of $L$ includes entries that are not 
zero modulo $p$.
The argument depends on the observation
that, as $k$ is alternating and has prime determinant, it is a prime knot.

As no column of $L$ is zero modulo $p$, there will be
heterogeneous colorings of $\P$, where a coloring
is {\it heterogeneous} if it assigns different colors to different arcs.
To complete the argument, we use the fact that the nullity of $C$ is one to
conclude that if one non-trivial coloring is heterogeneous, then they all are.

The structure of our paper is as follows.
In Section 2, we briefly discuss colorability and crossing matrices,
and present preliminary definitions and notation used throughout the
paper.  We also prove two lemmas, one showing that appending a 
zero to a vector in $C$'s null space gives a vector in $C'$'s null space, 
while the other demonstrates that taking the mirror reflection of an
alternating diagram corresponds to transposing the crossing matrix.
In Section 3, we introduce the coloring matrix $L$ and
develop the properties of colorings constructed from its columns. This culminates in a proof that every column of $L$ 
includes entries that are not zero modulo the determinant.
Finally, in Section 4, we prove the conjecture.

\section{Preliminaries}\label{sec-tools}
In this section, we review basic results on colorability; references for this 
material include \cite{livingston}, \cite{Mo}, and \cite{Przyt}.
We end the section with two lemmas.

Following~\cite{livingston}, a {\it diagram} of a knot is a planar representation
with gaps left to show where the knot crosses under itself; 
for example, see figure~\ref{fig:DtoG} below. Let $[k]$ denote
the set of diagrams of a knot $k$.
We say that a diagram is {\it reduced} if it has no nugatory crossings
(see \cite[Chapter 3]{Adams}, for example).
Let $[k_r]$ denote the subset of $[k]$ consisting of all reduced diagrams. 
We say $k$ is $n$-crossing if some $\P \in
[k]$ has $n$ crossings and no other $\P' \in [k]$ has fewer than $n$
crossings. It's known (by \cite{Kauff2, Mura2, Thistle}) that if $k$ is an $n$-crossing
alternating knot, then every $\P$ in $[k_r]$ is alternating and has
$n$ crossings and $n$ arcs.

We will now define {\it $p$-coloring} where
$p$ is an integer greater than $1$, $k$ is a knot, and $\P \in [k]$.
Let $x$, $y$, $z$ denote integers which label the over
arc and two under arcs, respectively, at a crossing of $\P$.  The crossing
satisfies the {\it condition of $p$-colorability} if
\begin{equation}\label{eq:modp}
2x - y - z \equiv 0 \mbox{ mod $p$.}
\end{equation}
We say $k$ is {\it p-colorable} if there is some $\P \in [k]$ such
that the arcs of $\P$ can be labeled, or colored, with the
numbers $0,\dotsc,p-1$ so that at least two numbers are used and
every crossing satisfies the condition of $p$-colorability. The
numbers $0,\dotsc,p-1$ are called {\it colors}.  The specific colors
assigned to the arcs make up a {\it $p$-coloring} of $\P$. A
$p$-coloring where every arc is assigned the same color is called
a {\it trivial coloring}.

Let $\P \in [k_r]$
be an $n$-crossing diagram of the knot $k$.  A {\it
labeling} of $\P$ is a particular indexing of the crossings
$\{c_1,\dotsc,c_n\}$ and of the arcs $\{a_1,\dotsc,a_n\}$. Given
$\P$ and some labeling, we define the {\it crossing matrix} $C'$ as
\begin{equation}\label{eq:C' matrix}
C'_{ij} = \begin{cases} 2 \ \mbox{if $a_j$ is the over arc at crossing $c_i$,}\\
-1 \ \mbox{if $a_j$ is an under arc at crossing $c_i$,}\\
0 \ \mbox{otherwise.}
\end{cases}
\end{equation}
Owing to \eqref{eq:modp} and \eqref{eq:C' matrix}, it follows that a
$p$-coloring can be represented by a vector $X'$ such that $C'X' \equiv
\vec 0$ mod $p$.

The matrix $C'$ provides a straightforward way to compute the
determinant of a knot.  In particular, let $C$ be any $(n-1) \times
(n-1)$ minor of $C'$.  We say $C$ is a {\it minor crossing matrix}.
The {\it determinant} of the knot $k$ is the absolute value of the
determinant of $C$: $\det k = |\det C|$. It follows that $k$ is
$p$-colorable if and only if $\gcd(p,\det k) > 1$. Given $C'$, there
are $n^2$ choices for $C$, but, to avoid ambiguity, we will reserve
the symbol $C$ for the minor crossing matrix obtained by removing
the last column and last row of $C'$. In this paper we will be
investigating $\det k$-colorings of diagrams of $k$ and, for simplicity, we will usually
just say ``coloring'' when the value $\det k$ is clear from the
context.

We conclude this section with two lemmas. The first shows that, by appending
a zero, a vector in the null space of $C$ can be ``extended" to a coloring.

\begin{lemma} \label{lemma:tac on a zero}
Let $k$ be a knot with reduced diagram $\D \in [k_r]$ and 
minor crossing matrix $C$. 
A vector $X$ such that $CX \equiv \vec 0$
mod $\det k$, can be extended to a coloring of the diagram $\D$
by adding a zero as the last entry. That is,
\begin{equation}\label{eq:tac on a zero }
C'\left(\begin{smallmatrix}X\\ 0\end{smallmatrix} \right)  \equiv \vec 0
\mbox{ mod $\det k$}.
\end{equation}
\end{lemma}

\begin{proof}
It is clear why the first $n-1$ entries
of the product in \eqref{eq:tac on a zero } should be $0$ mod $\det k$.
The last entry is also $0$ because the last
row of $C'$ can be expressed in terms of the first $n-1$ rows. Indeed, if
$k$ is an alternating knot and $r_i$ are the rows of its
crossing matrix then $r_j = - \sum_{i \ne j} r_i$.  More generally,
Livingston~\cite{livingston} shows that, for any
knot, each row of the crossing matrix can be expressed as a linear combination of
the other rows. Thus, the last entry of \eqref{eq:tac on a zero } is also zero.
\end{proof}

Our second lemma shows a connection between the crossing
matrices of two diagrams related by a mirror reflection.
Fix a specific diagram $\P \in [k_r]$  of an alternating knot $k$ and use the following
procedure to label its arcs and crossings:  After
orienting the knot, pick an arc, and label it $a_1$.
Following the orientation, label the next arc
$a_2$. The crossing that has $a_1$ and $a_2$ as under arcs
is labeled $c_1$. Continuing in this way, the $i$th arc we reach is labeled
$a_i$ and the crossing that has $a_{i-1}$ and $a_i$ as under arcs
is labeled $c_{i-1}$. Finally, the crossing between arcs $a_n$ and $a_1$ is labeled
$c_n$, where $n$ is the number of crossings in $\P$.
We call such a labeling an {\it oriented labeling}.

For an oriented labeling of an alternating diagram with $n$ arcs, the
expression $2a_j-a_i-a_{i+1}$, where $i<n$, which can be read off the
$i$th row of the crossing matrix, expresses that $a_j,a_i,a_{i+1}$
are the arcs present in crossing $c_i$.  For $i= n$, we have
$2a_j - a_1 - a_n$. Similarly,  $2c_l -
c_j - c_{j-1}$, for $j>1$, can be read off the $j$th column of the
crossing matrix, and expresses that the arc $a_j$ is an over
arc at the crossing $c_l$, and is an under arc at the
crossings $c_j$, $c_{j-1}$. For $j = 1$, the expression is $2c_l -
c_1 - c_n$.

Given $\P$, its {\it mirror image} $\P^m$ is the diagram obtained
by reversing all the crossings; that is, we change the over
crossings to under crossings and vice versa.  Oriented labelings of $\P$ correspond
to oriented labelings of $\P^m$:
for a crossing of $\P$ with over arc $a_j$ and under
arcs $a_i$, $a_{i+1}$, the corresponding crossing in $\P^m$ has
$a_i$ as the over arc, and $a_{j-1},a_j$ as the under arcs.
(Here we are taking indices modulo $n$, the crossing number. We will frequently
do this in what follows.)

In summary, the arc $a_i$ of $\P^m$ is an
over arc at crossing $c_{j-1}$, and, as always, $a_i$ is an under
arc at crossings $c_{i-1}$ and $c_i$. Thus, the crossing $c_i$ in $\P$
given by $2a_j - a_i - a_{i+1}$ transforms to the arc $a_i$ in
$\P^m$ given by $2c_{j-1} - c_{i-1} - c_{i}$.  More concisely, $a_j$
and $c_i$ in $\P$ correspond to $c_{j-1}$ and $a_i$ in $\P^m$,
respectively. This leads to the following lemma.

\begin{lemma}\label{mirroring lemma}
Let $k$ be an alternating knot.
Let $C'$ be a crossing matrix for $\P \in [k_r]$. Then the matrix
$C'^\intercal$ is a crossing matrix for $\P^m$.
\end{lemma}

\begin{proof}
Give $\P$ an oriented labeling.  Permute the rows of $C'^\intercal$
by sending the $i$th row to the $(i-1)$st row (for each $i > 2$) and the first row to
the $n$th row; call this matrix $D$. As in the discussion above, $D$ is the
crossing matrix for an oriented labeling of $\P^m$. However, the set
of crossing matrices of a diagram is closed under row and column
permutations, so $C'^\intercal$ is also a crossing matrix for
$\P^m$.
\end{proof}

\section{The Main Result on the Coloring Matrix}\label{sec:main result on L}

Let $k$ be an alternating knot of prime determinant and $\P \in [k_r]$.
Whichever $\P$ we choose, $|\det C| = \det k$, a knot invariant, 
so that $C$ has prime determinant and is therefore invertible over
the rationals.
We define $L = \det k \cdot C^{-1}$ and call $L$ the {\it coloring
matrix}. In this section we will prove that every column of $L$
contains entries that are not zero modulo $\det k$.

Our strategy is to argue by contradiction. We will show that if 
there is a column of zeroes, then there must be 
a ``pseudo coloring'' $Y'$. This is a way of labeling the arcs of $\D$ such
that the coloring condition~\eqref{eq:modp} fails at exactly two crossings, one, call it the $+1$--crossing, where 
$2x - y - z = 1$ and another, the $-1$--crossing, where $2x - y - z = -1$.
We then investigate the properties of $Y'$ eventually
deducing that $Y'$ exists only if $\P$ is the diagram
of a sum of two knots. (The distinguished $+1$-- and $-1$--crossings 
appear in distinct components.)
This is a contradiction; since $k$ is alternating and has prime determinant, it is in fact prime.

We begin with some of the properties of the coloring matrix. The
entries of $L$ are signed minors of $C$ and, therefore, integers.
Note that $CL$ and $LC$ both give the zero matrix mod $\det k$.
In particular, if $w_i$ is the $i$th column of $L$, then $Cw_i \equiv
\vec 0$ mod $\det k$.  We will use contradiction to show that 
$w_i \not\equiv \vec 0 \bmod \det k$.

\begin{lemma}\label{lemma:assume L has trivial column}
Let $k$ be an $n$-crossing alternating knot and let $\P \in [k_r]$.
If the $j$th column of $L$ is $\vec 0$ mod $\det k$, then there is a
pseudo coloring $Y'$ of $k$ such that $Y'$ has all positive entries
and $C'Y' = e_j - e_n$.
\end{lemma}

\begin{proof}
Let $L = \left(w_1 \cdots w_{n-1} \right)$ be the coloring matrix
of an alternating knot $k$.  For some $j<n$ we have $w_j \equiv \vec
0$ mod $\det k$. Equivalently, the $j$th column of $C^{-1}$ has all
integer entries. Let $Y$ be the $j$th column of $C^{-1}$, i.e.,  $Y =
\frac{1}{\det k}w_j$. Then, from the equation $CC^{-1}$ = $I$, we
infer that $CY = e_j$. Set $Y' = \left(\begin{smallmatrix} Y \\ 0
\end{smallmatrix}\right)$.  As in the proof of  lemma~\ref{lemma:tac on a zero}, we find that
\begin{equation} \label{eq:full integer relation}
C'Y' = e_j - e_n,
\end{equation}
where we now consider $e_j$ to be a vector of $\mathbb{R}^n$. Let $T$
denote the vector of all $1$'s, that is, a trivial coloring. To
ensure that $Y'$ has all positive entries we note that $C'T = \vec
0$, hence $C'(Y' + mT) = e_j - e_n$.  So, choose $m$ large
enough to ensure that $Y' + mT$ has all positive entries, and take that to be
$Y'$.
\end{proof}

\noindent \textbf{Remark:}
The pseudo coloring $Y'$ described by the lemma
is not a trivial coloring. A trivial coloring would give $C'Y' = \vec 0$.

\medskip

\begin{figure}[ht]
\begin{center}
\includegraphics[scale=0.5]{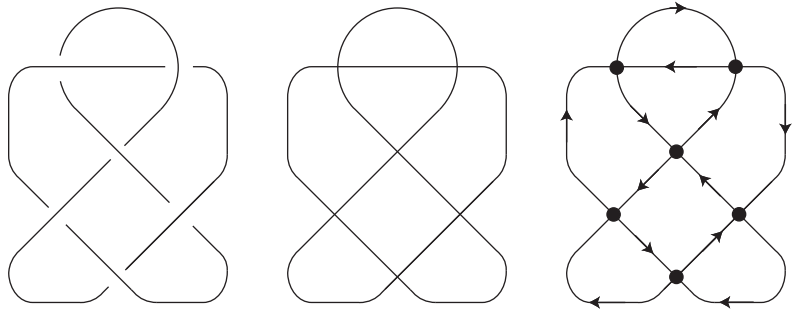}
\caption{From left to right: a diagram $\D$ of an alternating knot; ``filling in" the gaps; the graph $G_{\D}$.}\label{fig:DtoG}
\end{center}
\end{figure}

By ``filling in'' the gaps of the diagram $\D$ of an alternating knot, we have a regular projection of the knot $k$. We 
will view this as a four valent planar graph $G_{\D}$ by placing vertices at each crossing. 
As in figure~\ref{fig:DtoG}, we make a directed graph by orienting 
the ``under arc" edges away from the crossing vertex and the  ``over arc" edges towards the vertex. 
This choice of orientation will prove useful in what follows. 

In particular, it ensures that there is an Euler circuit.
Indeed, each vertex of  $G_{\D}$ has in-degree two
and out-degree two. Since $\D$ is the diagram of a knot, $G_{\D}$ is a connected graph and it follows  
that $G_{\D}$ has a directed Euler circuit, i.e., a closed path that runs through each edge exactly once, and passes through each vertex twice. (For example, see \cite[Theorem 1.4.24]{West}.)   Fix an Euler circuit $E$ in $G_{\D}$ and a starting edge $e_1$. Following $E$, the next edge will
be $e_2$ and so on. We then write
$$ E = (e_1, \ldots, e_{2n})$$
and subpaths of the Euler circuit will be denoted as subsequences of consecutive elements from the sequence above.

Let $Y'$ be a pseudo coloring as in lemma~\ref{lemma:assume L has trivial column}.
As we remarked above, $Y'$ is not a trivial coloring 
and therefore has a largest color $h$ and a smallest color $l$ with $0<l<h$.
We can carry the coloring of $\D$ over to a coloring of the edges of the digraph $G_{\D}$.
Call an edge with color $h$ an {\it $h$--edge}. 
Recall that there is a unique crossing, which we will call {\em the $+1$--crossing}, where $2x-y-z = 1$. Similarly, {\em the $-1$--crossing} will refer to the unique crossing where 
$2x - y - z = -1$.  Every other crossing is a {\em $0$--crossing}, i.e., 
$2x - y - z = 0$.
We will also call refer to the corresponding vertices of $G_{\D}$ as $+1$--, $0$--, or $-1$--crossings.

We will now investigate subpaths of the Euler circuit $E$ made
up exclusively of $h$--edges.

\begin{figure}[ht]
\begin{center}
\includegraphics[scale=0.8]{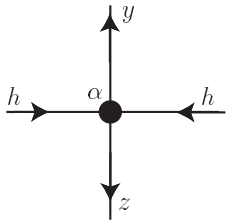}
\caption{Possible colors at an $\alpha$--crossing ($\alpha = -1$, $0$, or $1$) where an $h$--edge terminates. Set $y = h - \epsilon$. Since $\alpha = 2h - y - z$, then $z = h+ \epsilon + \alpha$.}\label{fig3}
\end{center}
\end{figure}

Let us pick an $h$--edge and follow $E$ starting with this
$h$--edge. The possible colors of edges at the vertex where the $h$--edge
ends are as illustrated in figure~\ref{fig3}. We may assign
one of the outgoing edges the color $h - \epsilon$ (with $\epsilon \geq 0$). The other 
will have color $h + \epsilon + \alpha$.
There are three possibilities:
\begin{itemize}
\item $-1$--crossing: over arc of color $h$, one of the under arcs of color $h - \epsilon$, and 
the other under arc of color $h + \epsilon + 1 > h$, which is impossible.
\item $0$--crossing: over arc of color $h$, one of the under arcs of color $h - \epsilon$, the other under arc of color  $h + \epsilon$, which works provided $\epsilon = 0$.
\item $+1$--crossing: over arc of color $h$, one of the under arcs of color $h - \epsilon$,
the other under arc of color $h + \epsilon -1$, which works for $\epsilon = 0$ or $\epsilon = 1$. These amount to the same situation:
over arc $h$; one of the under arcs $h$; and the other under arc $h - 1$.
\end{itemize}

We thus see that an $h$--edge points to the $+1$--crossing or to one of the
$0$--crossings, but never to the $-1$--crossing. Moreover, if it points
to a $0$--crossing then the other edges incident to this vertex will each
receive the color $h$ (we'll call such a $0$--crossing and the corresponding vertex in $G_{\D}$ an {\em $h$-vertex}); if it points to the $+1$--crossing then one of the 
outgoing edges receives color $h-1$ whereas the remaining edges receive
color $h$. In this way, following the Euler circuit $E$ starting from
an $h$--edge, one visits $h$-edges incident to $h$-vertices until one
reaches an edge incident to the $+1$--crossing. Here there are two
possibilities. Either the next edge is the $h-1$ edge and the
visiting of consecutive $h$--edges ends here or the next edge is
an $h$--edge and the visiting of consecutive $h$ edges continues.

We remark that any subpath of the Euler circuit which visits only
$h$-edges (we'll call it an {\em $h$--path} for short) and visits the $+1$--crossing twice,
must end at this crossing. As a matter of fact, its next edge
would be the $h-1$ edge incident to the $+1$--crossing. Consider
then the longest such $h$--path, call it $H$. 
By shifting the labels of the edges if necessary we have:
$$H = (e_1, \ldots , e_{n_H}) \subseteq (e_1, \ldots, e_{n_H}, e_{n_H+1}, \ldots, e_{2n}) = E.$$
Note that $e_{n_H}$ is necessarily directed towards the $+1$--crossing.

We now argue that each $h$--edge of $E$ is contained in $H$.

Assume to the contrary that $e_{n_H + l}$ is an $h$--edge. Then
resuming the Euler circuit $E$ as of this edge, either all remaining 
edges are $h$--edges -- which would imply that $e_{2n}$ is an $h$--edge
which would contradict $H$ being the longest $h$--path -- or there are 
edges other than $h$--edges between $e_{n_H+l}$ and $e_{2n}$. But
we have already noticed that the progression from an $h$--edge to an edge
with a less color goes through the unique $+1$--crossing vertex to its $h-1$ edge. However, this step has already been taken; it is $(e_{n_H}, e_{n_H+1})$. Since in an Euler circuit each edge is visited exactly once, then we
conclude that there are no $h$--edges past $e_{n_H}$.

Let us now observe that $e_1$ is the unique $h$--edge at its initial vertex
$v_1$. 
Indeed, the edges corresponding to the over arc cannot receive color $h$, for $H$ is the longest $h$--path, 
and it is straightforward to check that the other edge originating 
at $v_1$ is also not  an $h$--edge.  
Consequently, the vertices with outgoing $h$--edges are $v_1$, $h$--vertices, and the $+1$--crossing.



\begin{figure}[ht]
\begin{center}
\includegraphics[scale=0.4]{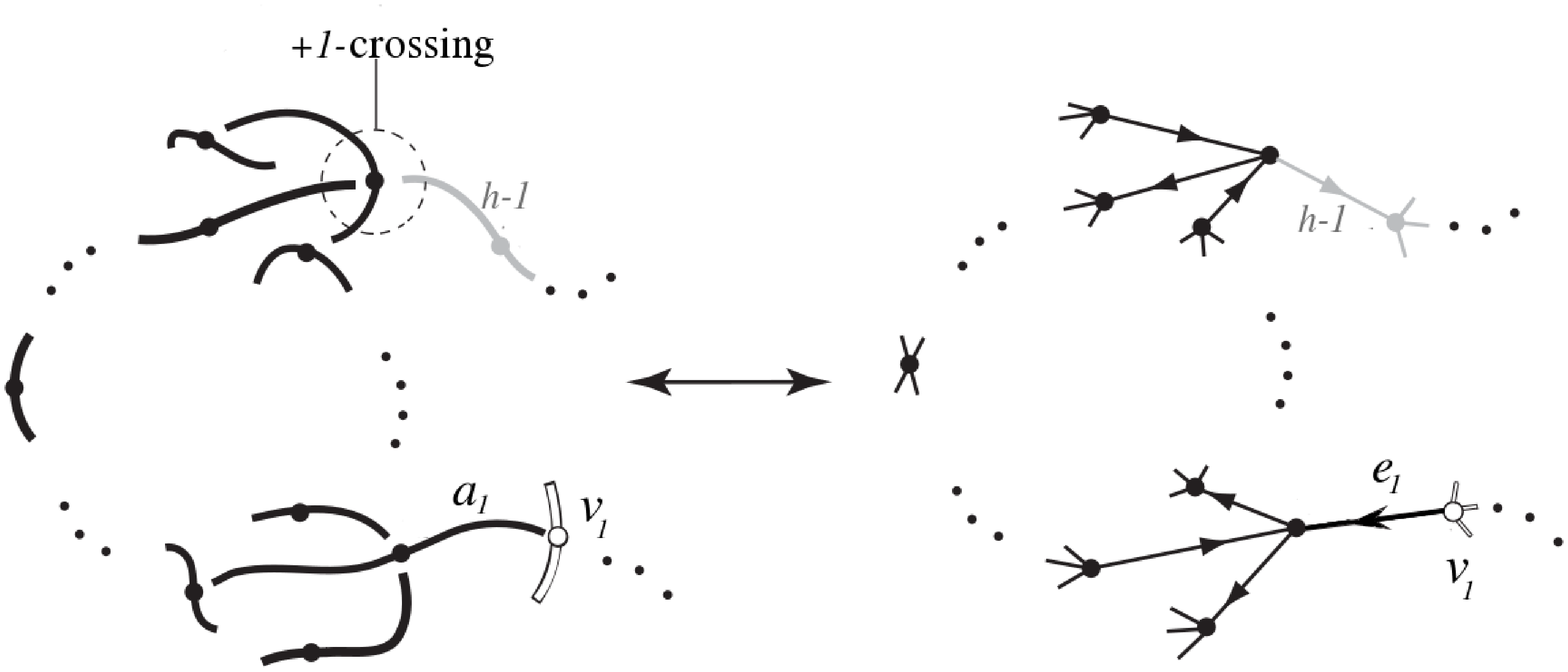}
\caption{The subpath $H$ of the Euler circuit (at right) and associated component $\H$ of $\D$ (at left).  All the black edges and arcs are $h$--edges and $h$--arcs.
}\label{fig:Ecircuit}
\end{center}
\end{figure}

Let us call $\H$ the portion of the knot diagram $\D$ which 
gives rise to the $h$--path $H$ of $E$. That is, $\H$ is the set of
{\em $h$--arcs} (i.e., arcs of color $h$) of $\D$ (see figure~\ref{fig:Ecircuit}). We will now argue that $\H$
factors out of $\D$ thereby showing that the knot under consideration
is not prime.  
The previous paragraph allows us to conclude that if we 
follow an orientation of $\D$ starting on an $h$--arc, then we remain 
on $h$--arcs until we reach either the $+1$--crossing or the crossing 
corresponding to $v_1$.  
Let $a_1$ be the arc of $\H$ which gives
rise to $e_1 \in H$ and choose an orientation of the knot so that 
starting at $a_1$ one progresses into $\H$.

Starting then at $a_1$ and following the orientation of the knot we 
eventually come back to $a_1$ so at some point we must have left $\H$. 
The only way of leaving $\H$ is through the $+1$--crossing
or the $v_1$ crossing. As we have chosen an orientation for $a_1$ 
directed away from the $v_1$ crossing, then,
reasoning along the same lines as we did 
for proving there were no $h$--edges outside $H$,
we only leave $\H$ once, at the $+1$--crossing, 

If this portion of the diagram (starting with $a_1$ and continuing until we leave $\H$) did not go over all the $h$--arcs of
$\H$ then starting at any one of the $h$--arcs left out we eventually come back to $a_1$. Only this time we never left $\H$ for the only 
crossings available are those corresponding to $h$-vertices and possibly the over arcs of the $+1$--crossing.  This implies the arc preceding $a_1$ is also an $h$--arc which in turn 
implies that the corresponding vertex $v_1$ in $G_{\D}$ has two $h$--edges stemming from it; but this possibility was ruled out above.  The contradiction shows that 
there are no $h$--arcs omitted when we start at $a_1$ and follow the 
orientation until we leave $\H$.
In other words, $\H$ corresponds to a summand of the knot under study. Moreover, by extending $\H$ just beyond the
$+1$--crossing if necessary, it is a summand that includes at least one crossing.
As $\D$ was a reduced diagram, if follows that $\H$ constitutes a
non-trivial summand. Repeating the reasoning above for 
the least color $l$ we conclude that there is a second non--trivial summand
distinct from the preceding one. We have thus proven the following
result.

\begin{prop} \label{lemma:D composite}
A knot with a diagram endowed with a $Y'$ pseudo coloring is not
a prime knot.
\end{prop}

Finally, we prove our main result on the coloring matrix.
\begin{prop}\label{prop:L has no triv. columns}
Let $k$ be an alternating knot with a prime determinant $p$ and $\P \in [k_r]$.
Then, every column of the coloring matrix $L$ includes entries that are not zero modulo $p$.
\end{prop}

\begin{proof}
Suppose, to the contrary, that some column of $L$ is zero modulo $p$.
Then, by  lemma \ref{lemma:assume L has trivial column}
there is a pseudo coloring $Y'$ and by proposition~\ref{lemma:D composite}
this implies $k$ is composite. 

However, an alternating knot of prime determinant is prime. To see this, let $k_1\#k_2$
denote the connected sum of knots $k_1$ and $k_2$, and let $\Delta_k(t)$ be the
Alexander polynomial of the knot $k$. Now, $\det k = |\Delta_k(-1)|$ and
$\Delta_{k_1 \# k_2}(t) = \Delta_{k_1}(t) \Delta_{k_2}
(t)$ (e.g., see \cite[Chap.\,6]{Murasugi}). From this it follows that
$\det (k_1 \# k_2) = \det k_1 \cdot \det k_2$. To complete the argument, note 
that the determinant of a non-trivial alternating knot is not one. 
For example, the determinant is bounded below by the crossing number
(see \cite[Proposition~13.30]{BZ}).
\end{proof}

\section{The Proof of the Conjecture}
In this section, we prove conjecture \ref{k-h conj}. Let $k$ be an
alternating knot of prime determinant and $\P \in [k_r]$. We will
say a coloring of $\P$ is {\it heterogeneous} if it assigns
different colors to different arcs. Thus, we can prove the
conjecture by showing that every nontrivial coloring of $\P$ is
heterogeneous. There are two steps in the argument. First we show
that if one coloring of $\P$ is heterogeneous, then they all are. We
conclude the argument by showing that there is such a heterogeneous
coloring.

For the first part, we introduce the idea of a fundamental coloring.
We say $\P$ has {\it one fundamental coloring} if given any two distinct nontrivial colorings $X_1$ and $X_2$
there are integers $a,b$ such that $X_2 \equiv aT + bX_1$ mod $\det k$
where $T$ is the trivial coloring, a vector of all $1$'s.
There are two immediate consequences.

\begin{lemma} \label{lemma:props of fundamental colorings}
Let $k$ have prime determinant and suppose
$\P \in [k]$ has one fundamental coloring. If a nontrivial coloring assigns
different colors to two particular arcs of $\P$, then
every nontrivial coloring will do so. Thus, if $\P$ has a heterogeneous
coloring, then every nontrivial coloring of $\P$ is heterogeneous.
\end{lemma}

So, in order to complete the first part of the argument, it will be enough
to show that a diagram of a knot of prime determinant has one 
fundamental coloring. 

\begin{prop}\label{prop:exists a fundamental coloring}
Let $k$ be a knot with prime determinant and $\P \in [k_r]$. 
Then $\P$ has one fundamental coloring.
\end{prop}

\begin{proof} 
Let $k$ have determinant $p$ and $\P$ have $n$ crossings. We will show that the colorings of $\P$ 
are determined by the colors of two specific arcs in $\P$.

Since $\det k = p$, the determinant of the minor crossing matrix is 
$\det C$ = $\pm p$.  With elementary row operations,
and the Euclidean algorithm for finding the gcd of two integers, we
can put $C$ in triangular form with integers on the diagonal. Its
determinant is then the product of the diagonal entries; all must be
$\pm 1$ except for one, which is $\pm p$. It follows that, mod $p$,
the rank of $C$ is $n-2$; we'll write $\dim \im C = n-2$.
Then the nullity, or dimension of the null space, is one; 
$\dim \ker C =1$.

This means that the dimension of the mod $p$ null space of the crossing matrix
$C'$ is two. Indeed, $\dim \im C' \ge \dim \im C = n-2$, so $\dim \ker C' \le 2$.  On the other hand, 
the trivial coloring assigning $1$ to every arc is not in $\ker C$ but is in $\ker C'$ mod $p$, so
$1 = \dim \ker C < \dim \ker C' \le 2$.
Therefore, the null space of $C'$ mod $p$ has dimension two.

We can now demonstrate that there is one fundamental coloring.
As the nullity of $C'$ is two, every coloring is determined by the coloring of two specific arcs of $\P$.
Consider two non-trivial colorings of $\P$, $X_1$ induced by coloring the two arcs $x_1$ and $y_1$ and $X_2$ induced by the colors $x_2$ and $y_2$. 
We can show that there is one fundamental coloring by finding $a$ and $b$ so that $X_2 \equiv a T + b X_1$.
Let  $$a \equiv  \frac{y_1x_2 - x_1 y_2}{y_1 - x_1} \quad \quad \mbox{ and } \quad \quad b \equiv \frac{y_2 - x_2}{y_1 - x_1}.$$
(Since $X_1$ is nontrivial, $y_1 \not\equiv x_1$ and $y_1-x_1$ has an inverse mod $p$.)  Then 
$$x_2 \equiv a \cdot 1 + b \cdot x_1 \quad \quad \mbox{ and } \quad \quad y_2 \equiv a \cdot 1 + b \cdot y_1,$$
as required.
\end{proof}

Combining lemma~\ref{lemma:props of fundamental colorings} and 
proposition~\ref{prop:exists a fundamental coloring} we have completed the first part of our argument:

\begin{cor}\label{cor:nullity corollary}
Let $k$ be a knot with prime determinant and $\P \in [k_r]$.  If
$\P$ has one heterogeneous coloring, then any nontrivial coloring is
heterogeneous.
\end{cor}

\noindent \textbf{Remark:}
The corollary applies not only to alternating knots, but to any knot
of prime determinant.

\medskip

It remains to show that there is a heterogeneous coloring.

\begin{prop}\label{prop:existence}
Let $k$ be an alternating knot with prime determinant.  If $\P \in
[k_r]$, then $\P$ admits a heterogeneous coloring.
\end{prop}
\begin{proof}
Let $k$ be an $n$-crossing alternating knot with prime determinant $p$. Construct an oriented labeling for $\P$
and the crossing matrix $C'$. Then, by lemma \ref{mirroring lemma},
$C'^\intercal$ is a crossing matrix for $\P^m$ and the minor crossing 
matrix for $\P^m$ is likewise $C^\intercal$. (Note that $\D^m$ is also
a reduced diagram of an alternating knot of determinant $p$.)
Let $L^m$ denote the coloring matrix of $\P^m$; elementary linear algebra shows $L^m = L^\intercal$.
Now, proposition \ref{prop:L
has no triv. columns} demonstrates that all the columns of
$L^\intercal$ have non-trivial entries mod $p$. So, for
example, in the first column of $L^\intercal$ there is a nonzero
entry: $(L^\intercal)_{i1} \not\equiv 0$.  That is, using lemma~\ref{lemma:tac on a zero},
the first column extends
to a coloring that distinguishes $a_i$ and $a_n$ in $\P^m$.

Then, by lemma~\ref{lemma:props of fundamental colorings} and proposition~\ref{prop:exists a fundamental coloring},
all nontrivial colorings distinguish $a_i$ and $a_n$ in $\P^m$. Hence
$(L^\intercal)_{ij} \not\equiv 0$ for $j = 1,\dotsc,n-1$. That is, the
$i$th row of $L^\intercal$, or equivalently, the $i$th column of
$L$ consists only of nonzero entries modulo $p$.  Therefore, the $i$th column
of $L$ extends to a coloring differentiating $a_n$ from all other
arcs in $\P$:
\begin{equation}\label{eq:coloring a_n ne a_i,all i}
X' = \left(
       \begin{array}{c}
         (L)_{i1} \\
         \vdots \\
         (L)_{i(n-1)} \\
         0 \\
       \end{array}
     \right)
\end{equation}
Now, since an oriented labeling can start at any arc, the arc
labeled with $a_n$ is arbitrary.  Hence, by repeating the above
argument while shifting the oriented labeling, we can exhibit a
coloring differentiating any arc from all other arcs.  But, by proposition~\ref{prop:exists a fundamental coloring}, $k$
has one fundamental coloring, so, by lemma~\ref{lemma:props of fundamental colorings}, every nontrivial coloring differentiates the
same arcs.  Therefore the coloring given in \eqref{eq:coloring
a_n ne a_i,all i} must be heterogeneous.
\end{proof}

Together, corollary \ref{cor:nullity corollary} and proposition
\ref{prop:existence} prove the conjecture.

\section*{Acknowledgement}

We thank the referee for many helpful comments that greatly improved
the exposition.

\end{document}